\documentclass[11pt]{article}

\usepackage{amsmath}
\usepackage{amsthm}
\usepackage{amssymb}
\usepackage{bussproofs}
\usepackage{stmaryrd}
\usepackage{geometry}
\usepackage{stix}
\usepackage{cleveref}
\usepackage{mathtools}

\usepackage{enumitem}

\newcommand{\TwoInt}{\mathrm{2Int}}
\newcommand{\NatTwoInt}{\mathrm{N2Int}}

\makeatletter
\DeclareRobustCommand{\rvdots}{%
  \vbox{
    \baselineskip4\p@\lineskiplimit\z@
    \kern-\p@
    \hbox{.}\hbox{.}\hbox{.}
  }}
\makeatother

\theoremstyle{plain}
\newtheorem{theorem}{Theorem}

\theoremstyle{definition}

\theoremstyle{remark}

\providecommand{\keywords}[1]
{
  \small	
  \textbf{\textit{Keywords---}} #1
}

\title{Strong Negation is Definable in $\TwoInt$}
\author{Hrafn Valt{\'y}r Oddsson\thanks{Institut für Philosophie I, Ruhr-Universit\"at Bochum.}}

\begin{document}

\maketitle

\begin{abstract}
I show that the strong negation is definable in $\TwoInt$, Wansing's bi-intuitionistic logic. 
\end{abstract}


\section{Preliminaries}

Wansing’s bi-intuitionistic logic $\TwoInt$ from \cite{Wansing2013} (see also \cite{WANSING2017}) is a constructive logic with two forms of derivation: proofs and dual proofs. The natural deduction system for $\TwoInt$ used here is a slight variant on Wansing’s original $\NatTwoInt$, but it is easily seen as equivalent.\footnote{The main difference is that we follow \cite{Ayhan2024} and use dashed lines in some of the rules to indicate that the conclusion can be obtained either by a proof or by a dual proof.} 

The language
$\mathcal{L}_{\TwoInt}$ of $\TwoInt$ is given as
follows:
$$A\Coloneqq p|\top|\bot|A\wedge A|A\vee A|A\to A|A\lefttail A|$$
 The connective $\lefttail$ is called the \emph{co-implication}. We use single lines for proofs and double lines for dual proofs. We write $[A]$ to indicate that the $A$ may be discharged as an assumption, and we write $\llbracket A \rrbracket$ to indicate that the $A$ may be discharged as a counter assumption. A dashed line represents two rules: one where all dashed lines are replaced with single lines and another where all dashed lines are replaced with double lines.

\begin{center}
\footnotesize\begin{tabular}{ccc}
\begin{minipage}{0.26\textwidth}
\begin{prooftree}
\AxiomC{}
\UnaryInfC{$A$}
\AxiomC{}
\UnaryInfC{$B$}
\RightLabel{$\wedge I^+$}
\BinaryInfC{$A \wedge B$}
\end{prooftree}
\end{minipage} &
\begin{minipage}{0.26\textwidth}
\begin{prooftree}
\AxiomC{}
\UnaryInfC{$A \wedge B$}
\RightLabel{$\wedge E_1^+$}
\UnaryInfC{$A$}
\end{prooftree}
\end{minipage} &
\begin{minipage}{0.26\textwidth}
\begin{prooftree}
\AxiomC{}
\UnaryInfC{$A \wedge B$}
\RightLabel{$\wedge E_2^+$}
\UnaryInfC{$B$}
\end{prooftree}
\end{minipage} \\[3em]

\begin{minipage}{0.26\textwidth}
\begin{prooftree}
\AxiomC{}
\UnaryInfC{$A$}
\RightLabel{$\vee I_1^+$}
\UnaryInfC{$A \vee B$}
\end{prooftree}
\end{minipage} &
\begin{minipage}{0.26\textwidth}
\begin{prooftree}
\AxiomC{}
\UnaryInfC{$B$}
\RightLabel{$\vee I_2^+$}
\UnaryInfC{$A \vee B$}
\end{prooftree}
\end{minipage} &
\begin{minipage}{0.26\textwidth}
\begin{prooftree}
\AxiomC{}
\UnaryInfC{$A \vee B$}
\AxiomC{$[A]$}
\noLine
\UnaryInfC{$\rvdots$}
\dashedLine
\UnaryInfC{ $C$ }
\AxiomC{$[B]$}
\noLine
\UnaryInfC{$\rvdots$}
\dashedLine
\UnaryInfC{ $C$ }
\RightLabel{$\vee E^+$}
\dashedLine
\TrinaryInfC{$C$}
\end{prooftree}
\end{minipage} \\[3em]

\begin{minipage}{0.26\textwidth}
\begin{prooftree}
\AxiomC{$[A]$}
\noLine
\UnaryInfC{$\rvdots$}
\UnaryInfC{$B$}
\RightLabel{$\to I^+$}
\UnaryInfC{$A \to B$}
\end{prooftree}
\end{minipage} &
\begin{minipage}{0.26\textwidth}
\begin{prooftree}
\AxiomC{}
\UnaryInfC{$A \to B$}
\AxiomC{}
\UnaryInfC{$A$}
\RightLabel{$\to E^+$}
\BinaryInfC{$B$}
\end{prooftree}
\end{minipage} &
\begin{minipage}{0.26\textwidth}
\end{minipage} \\[3em]

\begin{minipage}{0.26\textwidth}
\begin{prooftree}
\AxiomC{}
\UnaryInfC{$A$}
\AxiomC{}
\doubleLine
\UnaryInfC{$B$}
\RightLabel{$\lefttail I^+$}
\BinaryInfC{$A \lefttail B$}
\end{prooftree}
\end{minipage} &
\begin{minipage}{0.26\textwidth}
\begin{prooftree}
\AxiomC{}
\UnaryInfC{$A \lefttail B$}
\RightLabel{$\lefttail E_1^+$}
\UnaryInfC{$A$}
\end{prooftree}
\end{minipage} &
\begin{minipage}{0.26\textwidth}
\begin{prooftree}
\AxiomC{}
\UnaryInfC{$A \lefttail B$}
\RightLabel{$\lefttail E_2^+$}
\doubleLine
\UnaryInfC{$B$}
\end{prooftree}
\end{minipage} \\[3em]

\begin{minipage}{0.26\textwidth}
\begin{prooftree}
\AxiomC{}
\RightLabel{$\top I^+$}
\UnaryInfC{$\top$}
\end{prooftree}
\end{minipage} &
\begin{minipage}{0.26\textwidth}
\begin{prooftree}
\AxiomC{}
\UnaryInfC{$\bot$}
\RightLabel{$\bot E^+$}
\dashedLine
\UnaryInfC{ $A$ }
\end{prooftree}
\end{minipage} &
\begin{minipage}{0.26\textwidth}
\end{minipage} \\[2em]

\end{tabular}
\end{center}

\begin{center}
\footnotesize\begin{tabular}{ccc}

\begin{minipage}{0.26\textwidth}
\begin{prooftree}
\AxiomC{}
\doubleLine
\UnaryInfC{$A$}
\RightLabel{$\wedge I_1^-$}
\doubleLine
\UnaryInfC{$A \wedge B$}
\end{prooftree}
\end{minipage} &
\begin{minipage}{0.26\textwidth}
\begin{prooftree}
\AxiomC{}
\doubleLine
\UnaryInfC{$B$}
\RightLabel{$\wedge I_2^-$}
\doubleLine
\UnaryInfC{$A \wedge B$}
\end{prooftree}
\end{minipage} &
\begin{minipage}{0.26\textwidth}
\begin{prooftree}
\AxiomC{}
\doubleLine
\UnaryInfC{$A \wedge B$}
\AxiomC{$\llbracket A \rrbracket$}
\noLine
\UnaryInfC{$\rvdots$}
\dashedLine
\UnaryInfC{ $C$ }
\AxiomC{$\llbracket B \rrbracket$}
\noLine
\UnaryInfC{$\rvdots$}
\dashedLine
\UnaryInfC{ $C$ }
\RightLabel{$\wedge E^-$}
\dashedLine
\TrinaryInfC{$C$}
\end{prooftree}
\end{minipage} \\[3em]

\begin{minipage}{0.26\textwidth}
\begin{prooftree}
\AxiomC{}
\doubleLine
\UnaryInfC{$A$}
\AxiomC{}
\doubleLine
\UnaryInfC{$B$}
\RightLabel{$\vee I^-$}
\doubleLine
\BinaryInfC{$A \vee B$}
\end{prooftree}
\end{minipage} &
\begin{minipage}{0.26\textwidth}
\begin{prooftree}
\AxiomC{}
\doubleLine
\UnaryInfC{$A \vee B$}
\RightLabel{$\vee E_1^-$}
\doubleLine
\UnaryInfC{$A$}
\end{prooftree}
\end{minipage} &
\begin{minipage}{0.26\textwidth}
\begin{prooftree}
\AxiomC{}
\doubleLine
\UnaryInfC{$A \vee B$}
\RightLabel{$\vee E_2^-$}
\doubleLine
\UnaryInfC{$B$}
\end{prooftree}
\end{minipage} \\[3em]

\begin{minipage}{0.26\textwidth}
\begin{prooftree}
\AxiomC{}
\UnaryInfC{$A$}
\AxiomC{}
\doubleLine
\UnaryInfC{$B$}
\RightLabel{$\to I^-$}
\doubleLine
\BinaryInfC{$A \to B$}
\end{prooftree}
\end{minipage} &
\begin{minipage}{0.26\textwidth}
\begin{prooftree}
\AxiomC{}
\doubleLine
\UnaryInfC{$A \to B$}
\RightLabel{$\to E_1^-$}
\UnaryInfC{$A$}
\end{prooftree}
\end{minipage} &
\begin{minipage}{0.26\textwidth}
\begin{prooftree}
\AxiomC{}
\doubleLine
\UnaryInfC{$A \to B$}
\RightLabel{$\to E_2^-$}
\doubleLine
\UnaryInfC{$B$}
\end{prooftree}
\end{minipage} \\[3em]

\begin{minipage}{0.26\textwidth}
\begin{prooftree}
\AxiomC{$\llbracket B \rrbracket$}
\noLine
\UnaryInfC{$\rvdots$}
\doubleLine
\UnaryInfC{$A$}
\RightLabel{$\lefttail I^-$}
\doubleLine
\UnaryInfC{$A \lefttail B$}
\end{prooftree}
\end{minipage} &
\begin{minipage}{0.26\textwidth}
\begin{prooftree}
\AxiomC{}
\doubleLine
\UnaryInfC{$A \lefttail B$}
\AxiomC{}
\doubleLine
\UnaryInfC{$B$}
\RightLabel{$\lefttail E^-$}
\doubleLine
\BinaryInfC{$A$}
\end{prooftree}
\end{minipage} &
\begin{minipage}{0.26\textwidth}
\end{minipage} \\[3em]

\begin{minipage}{0.26\textwidth}
\begin{prooftree}
\AxiomC{}
\doubleLine
\UnaryInfC{$\top$}
\RightLabel{$\top E^-$}
\dashedLine
\UnaryInfC{ $A$ }
\end{prooftree}
\end{minipage} &
\begin{minipage}{0.26\textwidth}
\begin{prooftree}
\AxiomC{}
\RightLabel{$\bot I^-$}
\doubleLine
\UnaryInfC{$\bot$}
\end{prooftree}
\end{minipage} &
\begin{minipage}{0.26\textwidth}
\end{minipage}
\end{tabular}
\end{center}

A \emph{proof} of a formula $A$ from assumptions $\Gamma$ and counter assumptions $\Delta$ is a tree constructed from the rules above whose:
\begin{itemize}[noitemsep]
\item root is $\overline{A}$
\item leaves with single lines are either in $\Gamma$, discharged, or $\top$
\item leaves with double lines are either in $\Delta$, discharged, or $\bot$.
\end{itemize}
A \emph{dual proof} of a formula $A$ from assumptions $\Gamma$ and counter assumptions $\Delta$ is just as a proof of $A$ except its root is $\overline{\overline{A}}$.

We write $(\Gamma; \Delta) \vdash_{\TwoInt}^+ A$ when a proof of $A$ from assumptions $\Gamma$ and counter-assumptions $\Delta$, and $(\Gamma; \Delta) \vdash_{\TwoInt}^- A$ when there is a dual proof of $A$ from assumptions $\Gamma$ and counter-assumptions $\Delta$.

\section{Strong Negation in $\TwoInt$}
Strong negation was introduced in Nelson's constructive logic N4 \cite{Nelson1949, AlmukdadNelson1984}. In a bilateral framework, its role is to switch between proofs and dual proofs, which is captured by the following introduction and elimination rules from \cite[p. 28]{WANSING2017}:

\vspace{2em}

\begin{minipage}{0.24\textwidth}
\begin{prooftree}
\AxiomC{}
\doubleLine
\UnaryInfC{$A$}
\RightLabel{${\sim} I^+$}
\UnaryInfC{${\sim} A$}
\end{prooftree}
\end{minipage}
\begin{minipage}{0.24\textwidth}
\begin{prooftree}
\AxiomC{}
\UnaryInfC{$A$}
\RightLabel{${\sim} I^-$}
\doubleLine
\UnaryInfC{${\sim} A$}
\end{prooftree}
\end{minipage}
\begin{minipage}{0.24\textwidth}
\begin{prooftree}
\AxiomC{}
\UnaryInfC{${\sim} A$}
\RightLabel{${\sim} E^+$}
\doubleLine
\UnaryInfC{$A$}
\end{prooftree}
\end{minipage}
\begin{minipage}{0.24\textwidth}
\begin{prooftree}
\AxiomC{}
\doubleLine
\UnaryInfC{${\sim} A$}
\RightLabel{${\sim} E^-$}
\UnaryInfC{$A$}
\end{prooftree}
\end{minipage}

\vspace{2em}

\begin{theorem}
    The strong negation is definable in $\TwoInt$ via the formula $$(A \wedge (A \rightarrow (A \lefttail A))) \vee ((A \rightarrow A) \lefttail A).$$ 
\end{theorem} 

\begin{proof}
    We show that the rules for ${\sim}$ are derivable for $(A \wedge (A \rightarrow (A \lefttail A))) \vee ((A \rightarrow A) \lefttail A).$
\vspace{2em}

${\sim} I^+$:
\begin{prooftree}
\AxiomC{$[A]^1$}
\RightLabel{$\rightarrow\! I^+${(1)}}
\UnaryInfC{$A \rightarrow A$}
\AxiomC{}
\doubleLine
\UnaryInfC{$A$}
\RightLabel{$\lefttail I^+$}
\BinaryInfC{$(A \rightarrow A) \lefttail A$}
\RightLabel{$\vee I_2^+$}
\UnaryInfC{$(A \wedge (A \rightarrow (A \lefttail A))) \vee ((A \rightarrow A) \lefttail A)$}
\end{prooftree}

\vspace{2em}

${\sim} I^-$:
\begin{prooftree}
\AxiomC{}\UnaryInfC{$A$}
\AxiomC{$\llbracket A \rrbracket^1$}
\RightLabel{$\lefttail\! I^-${(1)}}
\doubleLine\UnaryInfC{$A\lefttail A$}
\RightLabel{$\rightarrow I^-$}
\doubleLine\BinaryInfC{$A \rightarrow (A \lefttail A)$}
\RightLabel{$\wedge I_1^-$}
\doubleLine\UnaryInfC{$A \wedge (A \rightarrow (A \lefttail A))$}
\AxiomC{}\UnaryInfC{$A$}
\AxiomC{$\llbracket A \rrbracket^2$}
\RightLabel{$\rightarrow I^-$}
\doubleLine\BinaryInfC{$A \rightarrow A$}
\RightLabel{$\lefttail\! I^-${(2)}}
\doubleLine\UnaryInfC{$(A \rightarrow A) \lefttail A$}
\RightLabel{$\vee I^-$}
\doubleLine\BinaryInfC{$(A \wedge (A \rightarrow (A \lefttail A))) \vee ((A \rightarrow A) \lefttail A)$}
\end{prooftree}

\vspace{2em}

${\sim} E^+$:
{\footnotesize
\begin{prooftree}
\AxiomC{}\UnaryInfC{$(A \wedge (A \rightarrow (A \lefttail A))) \vee ((A \rightarrow A) \lefttail A)$}
\AxiomC{$[A \wedge (A \rightarrow (A \lefttail A))]^1$}
\RightLabel{$\wedge E_1^+$}
\UnaryInfC{$A$}
\AxiomC{$[A \wedge (A \rightarrow (A \lefttail A))]^1$}
\RightLabel{$\wedge E_2^+$}
\UnaryInfC{$A \rightarrow (A \lefttail A)$}
\RightLabel{$\rightarrow E^+$}
\BinaryInfC{$A \lefttail A$}
\RightLabel{$\lefttail E_2^+$}
\doubleLine
\UnaryInfC{$A$}
\AxiomC{$[(A \rightarrow A) \lefttail A]^2$}
\RightLabel{$\lefttail E_2^+$}
\doubleLine
\UnaryInfC{$A$}
\RightLabel{$\vee E^+${(1,2)}}
\doubleLine
\TrinaryInfC{$A$}
\end{prooftree}
}

\vspace{2em}

${\sim} E^-$:
{\footnotesize
\begin{prooftree}
\AxiomC{}\doubleLine\UnaryInfC{$(A \wedge (A \rightarrow (A \lefttail A))) \vee ((A \rightarrow A) \lefttail A)$}
\RightLabel{$\vee E_1^-$}
\doubleLine\UnaryInfC{$A \wedge (A \rightarrow (A \lefttail A))$}
\AxiomC{$\llbracket A \rrbracket^1$}
\AxiomC{}\doubleLine\UnaryInfC{$(A \wedge (A \rightarrow (A \lefttail A))) \vee ((A \rightarrow A) \lefttail A)$}
\RightLabel{$\vee E_2^-$}
\doubleLine\UnaryInfC{$(A \rightarrow A) \lefttail A$}
\RightLabel{$\lefttail E^-$}
\doubleLine
\BinaryInfC{$A \rightarrow A$}
\RightLabel{$\rightarrow E_1^-$}
\UnaryInfC{$A$}
\AxiomC{$\llbracket A \rightarrow (A \lefttail A) \rrbracket^2$}
\RightLabel{$\rightarrow E_1^-$}
\UnaryInfC{$A$}
\RightLabel{$\wedge E^-${(1,2)}}
\TrinaryInfC{$A$}
\end{prooftree}
}
    
\end{proof}

\section*{Acknowledgments}
The question of strong negation's definability in $\TwoInt$ was posed to me by Sergey Drobyshevich, who had already shown that it is definable under the additional assumption of Peirce's law. My initial solution was then simplified by Satoru Niki, who also noted that $\top$ and $\bot$ are not needed. I also benefited greatly from conversations with Heinrich Wansing and Hitoshi Omori.

\bibliographystyle{alpha}
\bibliography{references}

\end{document}